\documentclass{amsart}
\usepackage{amsmath,amsfonts,amssymb}

\newtheorem{theorem}{Theorem}

\newtheorem{corollary}[theorem]{Corollary}

\newtheorem{definition}{Definition}

\newtheorem{lemma}[theorem]{Lemma}

\newcommand{\sgn}{\mathop{\mathrm{sgn}}}

\newcommand{\Pp}{\mathbf{P}}
\newcommand{\eps}{\epsilon}

\newcommand{\R}{\mathbb{R}}

\newcommand{\wt}{\widetilde}
\newcommand{\tauh}{\hat\tau_\epsilon}
\newcommand{\tauw}{\widetilde{\tau}_\epsilon}
\newcommand{\aone}{{\alpha_1^-}}
\newcommand{\atwo}{{\alpha_2^-}}
\newcommand{\Wc}{\mathcal{W}}

\newcommand{\sgm}{\tilde \sigma}

\title{Normal forms approach to diffusion near hyperbolic equilibria}
\author{Sergio Angel Almada Monter \and Yuri Bakhtin }
\address{School of Mathematics, Georgia Tech, Atlanta GA, 30332-0160, USA}
\email{salmada3@math.gatech.edu, bakhtin@math.gatech.edu}

\begin{document}

\begin{abstract}
We consider the exit problem for small white noise perturbation of a smooth dynamical
system on the plane in the neighborhood of a hyperbolic critical point. We show that
if the distribution of the initial condition has a scaling limit then the exit distribution
and exit time also have a joint scaling limit as the noise intensity goes to zero. The limiting 
law is computed explicitly. The result completes the theory of noisy heteroclinic networks in two
dimensions. The analysis is based on normal forms theory. 
\end{abstract}



\maketitle

\section{Introduction}
Small stochastic perturbations of continuous deterministic dynamical systems have been studied intensively for several decades. One of the
greatest achievements in the area is the celebrated Freidlin--Wentzell  (FW) theory that allows to explain long-term behavior of
systems with several meta-stable states at the level of large deviation estimates~\cite{Freidlin--Wentzell-book}.

An interesting situation where one can prove more precise estimates than those provided via FW quasi-potential approach was considered
by Kifer~\cite{Kifer}. He studied the exit problem for small noise perturbations of a deterministic system in a neighborhood of a hyperbolic fixed point (or, saddle)  in $\R^d$ assuming that
the starting point for the diffusion belongs to the stable manifold of the fixed point. Kifer showed that
as the noise level $\eps$ decays to 0, the diffusion tends to exit along the invariant manifold associated to the leading
eigenvalue $\lambda_+$ of the linearization of the system even in the presence of other unstable directions. He also found that the random exit time $\tau_\eps$ is
asymptotic in probability to $\lambda_+^{-1}\ln \eps^{-1}$.

When studying noisy perturbations of systems with heteroclinic networks, i.e.,\  multiple saddle points connected by heteroclinic orbits, 
Bakhtin \cite{nhn},\cite{nhn-ds}, realized that to understand the vanishing noise behavior of the system, one has to extend Kifer's work and analyze (i) 
the limiting distribution of the approximation error $\tau_\eps-\lambda_+^{-1}\ln \eps^{-1}$;
(ii) the limiting scaling laws of
the exit distribution for the neighborhood of each saddle. In fact, the exit distribution for the first
saddle point serves as the entrance distribution for the next saddle point, so that the peculiarities of the exit distribution can significantly 
influence the further evolution of the system.

The detailed analysis of scaling limits for distributional Poincar\'e maps near saddle points carried out in \cite{nhn} resulted in a complete theory for 
noisy heteroclinic networks. This theory explains interesting non-Markovian limit effects and the emerging patterns in the winnerless competion in the
process of sequential decision making (here, we are using the terminology from~\cite{Biology} where applications of heteroclinic networks to neural dynamics 
are considered).
The main result is that under the logarithmic time scaling the diffusion process converges in distribution in a special topology
to a precisely described limiting process that jumps between the saddles along the heteroclinic connections.

The core result that was applied in \cite{nhn} iteratively for sequences of saddle points connected to one another, is a lemma that computes the asymptotic 
scaling of the exit distribution for a neighborhood of a saddle point given the scaling of the entrance distribution.
The proof of that lemma was based on a coordinate change conjugating the driving drift vector field to a linear vector field. Although this method
and the lemma based on it apply in a fairly generic situation
where the so called no-resonance condition holds,
there are interesting cases such as Hamiltonian dynamics where the smooth linearization is not possible due to 
resonances. In these cases, the system remains nonlinear even under the optimal smooth change of coordinates, but
it has a certain special structure that can be studied using the classical theory of normal forms (see, 
e.g.,~\cite{MR1290117},\cite{IIlyashenko},\cite{Perko}).

In this paper,
 we extend the key lemma of \cite{nhn} to cover the resonant cases and, in fact, to the complete generality in the case $d=2$. Our approach is based on normal forms  that
have particularly nice structure in the 2-dimensional case. We 
believe that the main result of the present paper can be extended to higher dimensions.

An important consequence of our result is that in 2~dimensions the no-resonance restriction is completely removed from the theory of noisy heteroclinic networks developed in \cite{nhn},
so that the theory applies to any heteroclinic networks generated by smooth vector fields on the plane. It also provides a generalization of \cite{Bakhtin-SPA:MR2411523} and \cite{Kifer} in 2 dimensions. 



The structure of the paper is the following. In Section~\ref{Sec: Notation} we introduce the setting. In
Section~\ref{sec:statement_main} we state the main theorem and split the proof into several parts. 
In Section~\ref{sec: Normal_Forms} we introduce a simplifying change of coordinates in a small neigborhood of the saddle point.  The analysis of the transformed process in Section~\ref{sec: Main}
is based upon two results. Their proofs are given in Sections~\ref{sec:close_to_saddle} 
and~\ref{sec:unstable_manifold}.

$\mathbf{Acknowledgments.}$  The authors are grateful to Vadim Kaloshin for his advice on normal forms and for
pointing to \cite{IIlyashenko}. The work of Yuri Bakhtin is supported by NSF through a CAREER grant DMS-0742424.

\section{Setting}
\label{Sec: Notation}

Let us consider a $C^\infty$-smooth vector field~$b$ on $\R^2$ and a $C^2$-smooth matrix valued function $\sigma:\R^2 \to \R^{2 \times 2}$. 
Let $W$ be a standard $2$-dimensional Wiener process. In order to ensure that the stochastic It\^o
equation %
\begin{equation}
dX_\eps = b( X_\eps )dt + \eps \sigma(X_\eps)dW \label{eq:PrincipalEquation}
\end{equation}
has a unique global strong solution, our first
assumption is that both $b$ and $\sigma$ are Lipschitz and bounded, i.e.,\ there is a constant $L>0$ such that 
\begin{align*}
|\sigma(x)-\sigma(y)| \vee |b(x)-b(y)|& \leq L|x-y|,\quad x,y\in \R^2, \\
|\sigma(x)| \vee |b(x)|& \leq L,\quad x\in \R^2,
\end{align*}
where $|\cdot|$ denotes the Euclidean norm for vectors and Hilbert--Schmidt norm for matrices.
These conditions can be weakened, but we prefer this setting to avoid
multiple localization procedures throughout the text. For a general background on stochastic differential equations see, for example,~\cite{Karatzas--Shreve}. 

We shall denote by $S=(S^t)_{t\in\R}$ the flow generated by $b$: 
\begin{equation*}
\frac{d}{dt}S^{t}x=b(S^{t}x),\quad S^{0}x=x.
\end{equation*}%

Let $V$ be a domain in $\R^2$ with piecewise $C^2$ boundary.
We assume that the origin $0$ belongs to $V$ and it is a unique fixed point for $S$ in $\bar V$, 
or, equivalently, a unique critical point for $b$   in $\bar V$. Therefore, 
\begin{equation*}
b(x)=Ax+Q(x),
\end{equation*}%
where $A=Db(0)$ and $Q$ is the non-linear
part of the vector field satisfying $|Q(x)|=O(|x|^2)$, $x\to 0$. 

We assume that $0$ is a hyperbolic critical point, i.e.\ the matrix $A$ has two eigenvalues $\lambda_+$ and $-\lambda_-$ satisfiying $-\lambda_- < 0 < \lambda_+$. Without loss of generality, we suppose
that the canonical vectors are the eigenvectors for the matrix, so that $A=\mathop{\mathrm{diag}}(\lambda_+, -\lambda_-)$.

According to the Hadamard--Perron Theorem (see e.g.~\cite[Section~2.7]{Perko}),
the curves $\mathcal{W}^{s}$ and $\mathcal{W}^{u}$ defined via
\begin{equation*}
\mathcal{W}^{s}=\{x\in \mathbb{R}^{2}:|S^{t}x|\rightarrow 0\text{ as }%
t\rightarrow \infty \}
\end{equation*}%
and 
\begin{equation*}
\mathcal{W}^{u}=\{x\in \mathbb{R}^{2}:|S^{t}x|\rightarrow 0\text{ as }%
t\rightarrow -\infty \}
\end{equation*}%
are smooth, invariant under $S$ and tangent to $e_2$ and, respectively, to $e_1$ at $0$. The curve
$\Wc^{s}$ is called the stable manifold of $0$, and $\Wc^{u}$ is called the unstable manifold of $0$.

We assume that $\mathcal{W}^u$
intersects $\partial V$ transversally at points $q_{+}$ and $q_{-}$ such that
the segment of $\mathcal{W}^u$ connecting $q_{-}$ and $q_{+}$ lies entirely inside $V$ and 
contains~$0$.

We fix a point $x_{0}\in\mathcal{W}^s\cap V$ and
equip~\eqref{eq:PrincipalEquation} with the initial condition
\begin{equation}
X_{\epsilon }(0)=x_{0}+\epsilon ^{\alpha }\xi _{\epsilon },\quad\eps>0,
\label{eq:initial_condition}
\end{equation}%
where $\alpha \in (0,1]$ is fixed, and $(\xi _{\epsilon })_{\epsilon >0}$ is
a family of random vectors independent of $W$, such that for some random
vector $\xi _{0}$, $\xi _{\epsilon }\rightarrow \xi _{0}$ as $\epsilon
\rightarrow 0$ in distribution. 

If $\alpha\ne1$, then we impose a further technical
condition
\begin{equation}
 \Pp\{ \xi _0\parallel b(x_0)\} = 0,  \label{eqn: intial_hyp}
\end{equation}
where $\parallel$ denotes collinearity of two vectors.

We are studying the exit problem for the diffusion process $X_\eps$ in $V$. We are interested in the 
asymptotic distribution of the random point of exit of $X_\eps$ from~$V$ given by
$X_{\epsilon }(\tau _{\epsilon }^{V})$, where $\tau _{\epsilon }^{V}$ is the
stopping time defined by 
\begin{equation*}
\tau _{\epsilon }^{V}=\tau _{\epsilon }^{V}(x_{0})=\inf \{t>0:X_{\epsilon
}(t)\in \partial V\}.
\end{equation*}

\section{Main Result.}\label{sec:statement_main}
 The main result of the present paper is the following:

\begin{theorem}
\label{Thm: Main} In the setting described above, 
there is a family of random vectors $(\phi _{\epsilon })_{\eps>0}$, a family of random variables $(\psi_{\epsilon })_{\eps>0}$,
and a number 
\begin{equation}
\beta=\left\{ 
\begin{array}{cc}
1, & \alpha \lambda _- \geq \lambda_+ \\ 
\alpha \frac{\lambda_-}{\lambda _+}, & \alpha \lambda _- <\lambda _+ \\
\end{array}%
\right.
\label{eq:alpha_0}
\end{equation}
such that%
\begin{equation*}
X_\eps (\tau _\eps^{V})=q_{\sgn(\psi_{\epsilon })}+\epsilon
^{\beta}\phi _{\epsilon }.
\end{equation*}%
The random vector 
\[\Theta_\eps=\left(\psi_\eps, \phi_\eps, \tau _{\epsilon }^{V}+\frac{\alpha }{\lambda _{+}}\ln \epsilon\right)\]
converges in distribution as $\epsilon\to 0$.
\end{theorem}

The distribution of $\psi_{\epsilon }$,$\phi _{\epsilon }$, and the distributional limit of $\Theta_\eps$
will be described precisely.

The proof of Theorem~\ref{Thm: Main} has essentially three parts involving the analysis of diffusion (i) along $\mathcal{W}^s$;
(ii) in a small neighborhood of the origin; (iii) along~$\mathcal{W}^u$.  

The first part is based on a Theorem borrowed from \cite[Lemma 9.2] {nhn}.
To state the theorem, we need to introduce $\Phi _{x}(t)$ as the linearization of $S$ along the
orbit of $x\in\R^2$, i.e.\ we define $\Phi_{x}(t)$ to be the
solution to the matrix ODE
\begin{equation*}
\frac{d}{dt}\Phi _{x}(t)=A(t)\Phi _{x}(t)\text{, \ }\Phi _{x}(0)=I,
\end{equation*}
where $A(t)=Db(S^tx)$.
The theorem reads as:

\begin{theorem} 
\label{thm: far}
Let $x \in \R^2$ and $\left( \xi_\eps \right)_{\eps>0}$ be a family of random vectors independent of $W$ and convergent in distribution, as $\eps \to 0$, to $\xi_0$. Suppose
$\alpha\in(0,1]$ and let $X_\eps$ be 
the solution of the SDE~\eqref {eq:PrincipalEquation} with initial condition $X_{\epsilon }(0)=x+\epsilon ^{\alpha }\xi_{\epsilon }$. Then, for every $T>0$, the following representation holds true:
\begin{equation*}
X_{\epsilon }(T)=S^{T}x+\epsilon ^{\alpha }\bar{\xi}_{\epsilon },\quad \eps>0,
\end{equation*}%
where 
\begin{equation*}
\bar{\xi}_{\epsilon }\overset{\mathop{Law}}{\longrightarrow }\bar{\xi}_{0},\quad \eps\to 0,
\end{equation*}%
with%
\begin{equation*}
\bar{\xi}_{0}=\Phi _x(T)\xi_{0}+\mathbf{1}_{\{\alpha =1\}}N,
\end{equation*}%
$N$ being a Gaussian vector:
\begin{equation*}
N=\Phi _x (T)\int_{0}^{T}\Phi _x (s)^{-1}\sigma(S^s x)dW(s).
\end{equation*}
If $\alpha=1$ or assumption~\eqref{eqn: intial_hyp} holds, then $\Pp\{\bar \xi_0\parallel b({S^Tx})\}= 0$.
\end{theorem}

\medskip

The second part of the analysis is the core of the paper. Theorem~\ref{Thm: main_small} below
describes the behavior of the process in a small neighborhood $U$
of the origin. Notice that
since $x_0\in \Wc^s$, one can choose~$T$ large enough to ensure that that $S^Tx_0\in \Wc^s\cap U$. Therefore,
 the conditions of the following result are met if we use the terminal distribution of Theorem~\ref{thm: far} (applied to
the initial data given by~\eqref{eq:initial_condition}) as the initial
distribution.

\begin{theorem}\label{Thm: main_small}
There are two neighborhoods of the origin $U\subset U'\subset V $, two positive numbers $\delta<~\delta'$, 
and $C^2$ diffeomorphism $f:U'\to (-\delta',\delta')^2$, such that $f(U)=(-\delta,\delta)^2$
and the following property holds:
 
Suppose $x\in \Wc^s\cap U$, and $(\xi_\eps)_{\eps>0}$ is a family of random variables independent of $W$ and convergent in distribution, as $\eps\to0$, to $\xi_0$, where $\xi_0$ satisfies~\eqref{eqn: intial_hyp} with respect to $x$. Assume that
$\alpha\in(0,1]$ and that $X_\eps$ solves~\eqref{eq:PrincipalEquation} with initial condition~
\begin{equation}
X_\eps(0)=x+\eps^\alpha\xi_\eps,
\label{eq:small_entrance_scaling}
\end{equation} 
where $\xi_\eps$ satisfies condition~\eqref{eqn: intial_hyp} with respect to $x$.

There is also a family of random vectors $(\phi'_\eps)_{\eps>0}$, and a family of random variables $(\psi'_\eps)_{\eps>0},$ such that
\begin{equation*}
X_{\epsilon }(\tau _{\epsilon }^{U}) = g( \sgn(\psi'_\eps)\delta e_1)+\epsilon ^{\beta} \phi' _\eps,
\end{equation*}
where $g=f^{-1}$, $\beta$ is defined in~\eqref{eq:alpha_0}, and 
the random vector 
\[\Theta'_\eps=\left(\psi'_\eps, \phi'_\eps, \tau_{\epsilon }^{U}+\frac{\alpha }{\lambda _{+}}\ln \epsilon\right)\]
converges in distribution as $\epsilon\to 0$.
\end{theorem}

The notation for $\Theta'_\eps$ and its components is chosen to match the notation involved in the
statement of Theorem~\ref{Thm: Main}. 
Random elements $\psi' _{\epsilon }$,$\phi'_{\epsilon }$ and the distributional limit of $\Theta'_\eps$
will be described precisely, see~\eqref{eq:limiting_distr_for_Theta_prime}. 
Obviously, the symmetry or asymmetry in the limiting distribution of $\psi'_\eps$ results in the symmetric or asymmetric choice of exit direction so that the exits in the positive
and negative directions are equiprobable or not. On the other hand, the limiting distribution of $\phi'_\eps$ determining the asymptotics of the exit point can also be symmetric
or asymmetric which results in the corresponding features of the random choice of the exit  direction at the next saddle point visited by the diffusion. 

In Section~\ref{sec: Main} we prove Theorem~\ref{Thm: main_small} using the approach based on normal forms.

The last part of the analysis is devoted to the exit from $V$ along $\Wc^u$. We need the following statement which is a specific case of the main result of~\cite{Levinson}.

\begin{theorem}\label{Thm:Transversal_exit} In the setting of Theorem~\ref{thm: far}, assume additionally that (i) $q=S^Tx\in \partial V$; (ii) there is no
$t\in[0,T)$ with $S^tx\in\partial V$;  (iii) $b(q)$ is tranversal (i.e.\ not tangent) to~$\partial V$ at~$q$. Then
\begin{equation}
\tau_\eps^V\stackrel{\Pp}{\to} T,\quad \eps\to 0, 
\end{equation}
and
\begin{equation}
\eps^{-\alpha} (X_\eps(\tau_\eps^V)-q) \stackrel{Law}{\to}\pi \bar \xi_0 ,\quad \eps\to0,
\end{equation}
where $\pi$ denotes the projection along $b(q)$ onto the tangent line to $\partial V$ at $q$. 
\end{theorem}

Now Theorem~\ref{Thm: Main} follows from the consecutive application of Theorems~\ref{thm: far} through~\ref{Thm:Transversal_exit} and with 
the help of the strong Markov property. In fact, in this chain of theorems, the conclusion of 
Theorem~\ref{thm: far} ensures that the conditions of Theorem~\ref{Thm: main_small} hold, and
the conclusion of the latter ensures that the conditions of Theorem~\ref{Thm:Transversal_exit} hold.
Notice that
the total time needed to exit~$V$ equals the sum of times described in the three theorems.
Notice also that at each step we can compute the limiting initial and terminal distributions explicitly. Theorems~\ref{thm: far} and~\ref{Thm:Transversal_exit} contain the respective formulas 
in their formulations, and the explicit limiting distribution for $\Theta'_\eps$ of Theorem~\ref{Thm: main_small} is computed in~\eqref{eq:limiting_distr_for_Theta_prime}.

\section {Simplifying change of coordinates}
\label{sec: Normal_Forms}
In this section we start analyzing the diffusion in the neighborhood of the saddle point. The first step is to
find a smooth coordinate change that would simplify the system. This can be done with the help of the theory of normal forms.

Let $g$ be a $C^\infty$ diffeomorphism of a neighborhood of the origin with inverse~$f$. When $X_\eps$ is close to the origin and belongs to the image of that neighborhood under $g$,
we can  use It\^o's
formula to 
see that $Y_\eps=f(X_\eps)$ satisfies 
\begin{align*}
dY_{\epsilon } &=Df(X_{\epsilon })dX_{\epsilon }+\frac{1}{2}%
[Df(X_{\epsilon }),X_{\epsilon }] \\
&=Df(g (Y _{\epsilon }))b(g (Y_{\epsilon
}))dt+\epsilon \sgm (Y_{\epsilon })dW +\epsilon ^{2}\Psi (Y_{\epsilon })dt,
\end{align*}%
for some smooth function $\Psi :\mathbb{R}^{2}\rightarrow \mathbb{R}^{2}$
and $\sgm =\left ((Df)\circ g \right) \sigma.$ 
Here $D$ denotes the Jacobian matrix, and the square brackets mean
quadratic covariation.
Since $Df \circ g =(Dg)^{-1}$, we can rewrite the above SDE as
\begin{equation}
dY_{\epsilon }=\left( \left( Dg (Y_{\epsilon })\right)
^{-1}b(g (Y_{\epsilon }))+\epsilon ^{2}\Psi (Y_{\epsilon
})\right) dt+\epsilon \sgm (Y_{\epsilon })dW.  \label{eqn: Eqn_zeta}
\end{equation}

The idea now is to choose a transformation $g$ (or, equivalently, $f$) that makes the drift
in equation~(\ref{eqn: Eqn_zeta}) easy to estimate. We are going to use the normal form theory and so we need to recall certain terminology, notation
and results from~\cite{IIlyashenko} putting them in the (two-dimensional) context of this paper.

A pair of complex numbers $\lambda=(\lambda_1, \lambda_2)$ is said to be non-resonant if there are no 
integral relations between them of the form $\lambda_j=\alpha \cdot \lambda$, where $\alpha =(\alpha _1,\alpha_2) \in \mathbb{Z}_{+}^{2}$ 
is a multi-index with $|\alpha|=\alpha_1 + \alpha_2 \geq 2$. Otherwise, we say that it is resonant. Moreover, a resonant $\lambda$ is said to be one-resonant 
if all the resonance relations for $\lambda$ follow from a single resonance relation.  A monomial $x^\alpha e_j=x_1^{\alpha _1 } x_2 ^{\alpha _2 } e_j$ 
is called a resonant monomial of order~$R$ if~$\alpha \cdot \lambda = \lambda_j$ and $|\alpha|=R$. Normal form theory asserts (see~\cite{IIlyashenko},\cite{MR1290117}) 
that for any pair of integers $R\geq1$ and $k \geq 1$, there are two neighborhoods of the origin $\Omega_f$ and $\Omega_g$  and a $C^k$-diffeomorphism 
$f:\Omega_f\to \Omega_g$ with inverse  $g:\Omega_g\to \Omega_f$
such that  
\begin{equation}
\left( Dg (y) \right)^{-1}b( g (y) )=Ay + P(y) + \mathcal{R}(y), \quad y \in \Omega_g \label{eqn; Pre-NormalForm}
\end{equation} 
where $P$ is a polynomial containing only resonant monomials of order at most~$R$ and $\mathcal{R}(\zeta)=O(|\zeta|^{R+1})$.
 If $\lambda$ is non-resonant, then $f$ can be chosen so that both $P$ and $\mathcal{R}$ in~\eqref{eqn; Pre-NormalForm} are identically zero.  
Moreover, due to~\cite[Theorem~3,Section~2]{IIlyashenko},  if $\lambda$ is one-resonant then $f$ can be chosen so that $\mathcal{R}$ in~\eqref{eqn; Pre-NormalForm} is identically zero. 
More precisely, if $\lambda$ is a one-resonant pair, then for any pair of integers $R\geq1$ and $k\geq1$, 
there are two neighborhoods of the origin $\Omega_f$ and $\Omega_g$  and a $C^k$-diffeomorphism 
$f:\Omega_f\to \Omega_g$ with inverse  $g:\Omega_g\to \Omega_f$
such that  
\begin{equation}
\left( Dg(y) \right)^{-1}b( g (y) )=Ay + P(y), \quad y \in \Omega_g, \label{eqn: NormalForm}
\end{equation}
where $P$ is a polynomial that contains only resonant monomials.

Note that $(\lambda_+,-\lambda_-)$ is either non-resonant or one-resonant (resonant cases that are not one-resonant are possible in higher dimensions where
pairs of eigenvalues get replaced by vectors of eigenvalues). The non-resonant case (in any dimension) was studied in~\cite{nhn}. 
In this paper, we extend the analysis of~\cite{nhn} to the non-resonant case, i.e.\ the one-resonant case, given that we are working in 2 dimensions. 

To find all resonant monomials of a given order $r\ge 2$, we have to find all the integer solutions to the two $2 \times 2$ systems of equations:%
\begin{align*}
\alpha _{1}\lambda _{+}-\alpha _{2}\lambda _{-} &=\pm\lambda _{\pm}, \\
\alpha _{1}+\alpha _{2} &=r.
\end{align*}
Therefore, the power multi-indices of a resonant monomial of order $r$ has to coincide with one of the following:
\begin{eqnarray}
(\alpha _{1}^{+}(r),\alpha _{2}^{+}(r)) &=&\frac{1}{\lambda _{+}+\lambda _{-}%
}(\lambda _{+}+r\lambda _{-},(r-1)\lambda _{+}),
\label{eqn: alpha_plus} \\
(\alpha _{1}^{-}(r),\alpha _{2}^{-}(r)) &=&\frac{1}{\lambda _{+}+\lambda _{-}%
}((r-1)\lambda _{-},r\lambda _{+}+\lambda _{-}),  \label{eqn: alpha_minus}
\end{eqnarray}%
Let us make some elementary observations on integer solutions of these equations for $r\ge 2$.
\begin{enumerate}
\item None of the solution indices can be $0$.  Moreover, neither $\alpha_1^+(r)$ nor $\alpha_2^-(r)$ can be equal to $1$. 
\item As functions of $r$, $\alpha_i^\pm (r)$ are increasing.
\item Expressions~\eqref{eqn: alpha_plus} and~\eqref{eqn: alpha_minus} cannot be an integer for $r=2$. 
\item The term $P=(P_1,P_2)$ in~\eqref{eqn: NormalForm} satisfies $P_1(y)=O(y_1^2 |y_2|)$ and~$P_2(y)=O(|y_1| y_2^2)$. 
This observation is a consequence of observations 1 and 3 since they imply that resonant multi-indices have to satisfy $\alpha^+(r)\ge (2,1)$ and $\alpha^-(r)\ge (1,2)$ coordinatewise.
\item If at least one of the coordinates $y_1$ and $y_2$ is zero, then  $P(y_1,y_2)=0$. This is a direct consequence of the previous observation.
\end{enumerate}
Given all these considerations, the main theorem of this section is a simple consequence of~\cite{IIlyashenko}.
\begin{theorem}
\label{Lemma: Def_H1&H2} In the setting described in Section~\ref{Sec: Notation},
there is a number $\delta'>0$, a neighborhood of the origin $U'$, and a $C^2$-diffeomorphism $f:U'\to (-\delta',\delta')$
with inverse  $g:(-\delta',\delta')^2\to U'$ such that  the following property holds.

If $X_\eps(0)\in U$, then the stochastic process $Y_\eps=(Y_{\eps,1},Y_{\eps,2})$ given by 
\[
Y_\eps(t)=f (X_\eps(t\wedge \tau_\eps^{U}))
\]
satisfies the following system of SDEs up to $\tau_\eps^{U}$ : 
\begin{align}
\label{eq:SDE_changed_coord1}
dY_{\epsilon,1 } &=\left(  \lambda _{+}Y_{\epsilon,1 }+H_{1}(Y_{\epsilon}, \eps) \right) dt+\epsilon \sgm_{1}(Y_{\epsilon})dW \\
\label{eq:SDE_changed_coord2}
dY_{\epsilon,2 } &=\left( -\lambda _{-}Y_{\epsilon,2 }+H_{2}(Y_{\epsilon}, \eps) \right) dt+\epsilon \sgm_{2}(Y_\epsilon)dW,
\end{align}
where $\sgm_i:(-\delta',\delta')^2 \to \R$ are $C^1$ functions for $i=1,2$. 
The functions $H_i$ are given by $H_i=\hat H_i + \eps^2 \Psi_i$, where 
$\Psi_i:(-\delta',\delta')^2 \to \R^2$ are continuous bounded functions, and  $\hat H_i: (-\delta',\delta')^2 \times [0,\infty)$ are polynomials, so that for some constant $K_{1}>0$ 
and for any $y\in (-\delta',\delta')^2$,
\begin{align*}
|\hat H_{1}(y)| &\leq K_{1}|y_1|^{\alpha_1^+}|y_2|^{\alpha_2^+}, \\
|\hat H_{2}(y)| &\leq K_{1}|y_1|^\aone |y_2|^\atwo.
\end{align*}
Here, the integer numbers $\alpha_i^\pm$, $i=1,2$, are such that $(\alpha_1^+,\alpha_2^+)$ is of the form ~\eqref{eqn: alpha_plus} 
for some choice of $r=r_1\ge 3$, and
and $(\alpha_1^-,\alpha_2^-)$ is of the form~\eqref{eqn: alpha_minus} for some choice $r=r_2\ge 3$. In particular,
\begin{align*}
|H_{1}(y,\eps)| &\leq K_{1}y_1^{2}|y_2|+K_{2}\epsilon ^{2},\\
|H_{2}(y,\eps)| &\leq K_{1}|y_1|y_2^{2}+K_{2}\epsilon ^{2},
\end{align*}
for some constants $K_{1}>0$ and $K_{2}>0$.
 
\end{theorem}

\section{Proof of Theorem~\ref{Thm: main_small}}
\label{sec: Main}
In this section we derive Theorem~\ref{Thm: main_small} from several auxiliary statements. Their proofs are postponed to later sections.

Theorem~\ref{Lemma: Def_H1&H2} allows to work with process $Y_\eps=f(X_\eps)$ instead of $X_\eps$ while $Y_\eps$ stays in $(-\delta',\delta')^2$

If we take $\delta\in(0,\delta')$, then for the initial conditions considered in Theorem~\ref{Thm: main_small} and given in~\eqref{eq:small_entrance_scaling},
\[
 \Pp\{X_\eps(0)\in U'\}\to 1,\quad\eps\to 0,
\]
i.e.,
\[
 \Pp\{Y_\eps(0)\in (-\delta',\delta')^2\}\to 1,\quad\eps\to 0.
\]
Moreover, denoting $f(x)$ by $y=(0,y_2)$ we can write
\[
 Y_\eps(0)=y+\eps^\alpha \chi_\eps = (\eps^\alpha \chi_{\eps,1},\ y_2+\eps^\alpha \chi_{\eps,2}),\quad \eps>0,
\]
where $\chi_\eps=(\chi_{\eps,1},\chi_{\eps,2})$ is a random vector convergent in distribution to $\chi_0=(\chi_{0,1},\chi_{0,2})=Df(x)\xi_0$. Due to the hypothesis in Theorem~\ref{Thm: main_small}, we notice that the distribution of $\chi_{0,1}$ has no atom at $0$.

Let us take any $p \in (0,1)$ such that 
\begin{equation}
 1-\frac{\lambda _{+}}{\lambda _{-}}<p<  \frac{\lambda_{-}}{\lambda_{+}+\lambda_{-}},  
\label{eqn: p_prop}
\end{equation}
and define the following stopping time: 
\[
\hat\tau _{\epsilon}=\inf \{t:|Y_{\epsilon,1 }(t)|=\epsilon^{\alpha p}\}.
\]

Up to time $\hat \tau_\eps$, the process $X_\eps$ mostly evolves along the stable manifold~$\mathcal{W}^{s}$.
After $\hat \tau_\eps$, it evolves mostly along the unstable manifold $\mathcal{W}^{u}$. Process $Y_\eps$ evolves
accordingly, along the images of $\Wc^s$ and $\Wc^u$ coinciding with the coordinate axes.

Let us introduce random variables $\eta_\eps^{\pm}$ via
\begin{align*}
\eta _{\epsilon }^{+}&=\epsilon ^{-\alpha }e^{-\lambda _{+}\hat\tau
_{\epsilon }}Y_{\epsilon,1 }(\hat\tau_{\epsilon }),\\ 
\eta
_{\epsilon }^{-}&=\epsilon ^{-\alpha (1-p)\lambda _{-}/\lambda
_{+}}Y_{\epsilon,2 }(\hat\tau _{\epsilon }).
\end{align*}
Also we define the distribution of random vector $(\eta _{0}^{+ },\eta_0^-)$ via
\begin{eqnarray}
\eta _{0}^{+} &=&\chi_{0,1}+\mathbf{1}_{\{\alpha =1\}}N^{+}, \label{eq:eta_0_+}\\
\eta _{0}^{-} &=&|\eta _{0}^{+}|^{\lambda _{-}/\lambda _{+}}y_{2},\notag
\end{eqnarray}%
where
\begin{equation}
N^+=\int_{0}^{\infty}e^{-\lambda _-s}\sgm_{1}(0,e^{-\lambda _{-}s}y_{2})dW
\label{eq:N_0_+}
\end{equation}
is independent of $\chi_{0,1}$.

\begin{lemma}
\label{prop: eta_convergence} If the first inequality in~\eqref{eqn: p_prop} holds, then
\begin{equation}
\label{eq:exit_at_tauh_along_axis_1} 
\Pp\{Y_{\eps,1}(\tauh) = \eps^{\alpha p} \sgn \eta^+_\eps\}\to 1,\quad \eps\to 0.
\end{equation}
and
\begin{equation}
 \left(\eta_\eps^+,\eta_\eps^-, \hat \tau_\eps+\frac{\alpha }{\lambda _{+}}(1-p)\log \epsilon\right)\ 
\stackrel{Law}{\longrightarrow}\ \left(\eta_0^+,\eta_0^-, -\frac{1}{\lambda _{+}}\log |\eta _{0 }^{+}| \right),\quad \eps\to 0. 
\end{equation}
\end{lemma}

We prove this lemma in Section~\ref{sec:close_to_saddle}. Along with the strong Markov property, it allows to reduce the study of the evolution of $Y_\eps$ after $\tauh$ to studying
the solution of system~\eqref{eq:SDE_changed_coord1}--\eqref{eq:SDE_changed_coord2} with initial condition
\begin{equation}
\label{eq:restart_at_tauh}
Y_{\epsilon }(0)=(\epsilon^{\alpha p}\sgn\eta _{\epsilon }^{+},\epsilon ^{\alpha (1-p)\lambda
_{-}/\lambda _{+}}\eta _{\epsilon }^{-}),
\end{equation}
where 
\begin{equation}
\label{eq:condition_for_theorem_along_unstable}
(\eta^+_\eps,\eta^-_\eps)\stackrel{Law}{\longrightarrow}(\eta^+_0,\eta^-_0),\quad\eps\to 0.
\end{equation}
We denote
\begin{equation}
 \tau _{\epsilon }=\tau _{\epsilon }(\delta )=\inf \{t\ge 0:|Y_{\epsilon,1}(t)|=\delta\}. \label{eqn: tau_def}
\end{equation}
Our next goal is to describe the behavior of $Y(\tau_\eps)$. To that end, we introduce a random variable $\theta$ via
\begin{equation}
\label{eq:theta}
\theta\stackrel{Law}{=}\begin{cases}
N, & \alpha \lambda _{-}> \lambda _{+}, \\
\left( \frac{|\eta _{0}^{+}|}{\delta }\right) ^{\lambda _{-}/\lambda
_{+}}y_{2}+N, & \alpha \lambda _{-}= \lambda _{+}, \\ 
\left( \frac{|\eta _{0}^{+}|}{\delta }\right) ^{\lambda _{-}/\lambda
_{+}}y_{2}, & \alpha \lambda _{-}< \lambda _{+}.%
\end{cases}
\end{equation} 
where the distribution of $N$ conditioned on $\eta_0^+$, on $\{\sgn \eta_0^+=\pm1\}$ is
centered Gaussian with variance
\begin{equation*}
\sigma _{\pm}=\int_{-\infty}^{0}e^{2\lambda _{-}s}
\left|\sgm_2(\pm \delta e^{\lambda _{+}s},0)\right|^2ds.
\end{equation*}%
Let us also recall that $\beta$ is defined in~\eqref{eq:alpha_0}.

\begin{lemma}\label{Thm:after_tauh} Consider the solution to system~\eqref{eq:SDE_changed_coord1}--\eqref{eq:SDE_changed_coord2}
equipped with initial conditions~\eqref{eq:restart_at_tauh} satisfying~\eqref{eq:condition_for_theorem_along_unstable}.
If the second inequality in~\eqref{eqn: p_prop} holds, then
\begin{equation}
\label{eq:exit_through_unstabale}
 \Pp\{|Y_{\eps,1}(\tau_\eps)|=\delta\}\to 1,\quad \eps\to 0,
\end{equation}
\begin{equation}
 \tau_\eps+\frac{\alpha p}{\lambda_+}\log\eps\stackrel{\Pp}{\longrightarrow}
\frac{1}{\lambda_+}\log\delta,
\label{eq:tau_eps_convergence}
\end{equation}
\begin{equation}
\epsilon ^{-\beta}Y_{\epsilon,2}(\tau _{\epsilon })\overset{Law}{\longrightarrow }\theta.
\label{eq:asymptotics_for_Y_2_tau}
\end{equation}%
Moreover, if $\beta<1$, then the convergence in probability also holds.
\end{lemma}

A proof of this lemma is given in Section~\ref{sec:unstable_manifold}.

Now Theorem~\ref{Thm: main_small} follows from Lemmas~\ref{prop: eta_convergence} and~\ref{Thm:after_tauh}.
In fact, the strong Markov property and~\eqref{eq:exit_at_tauh_along_axis_1} imply
\[
\Pp\{\tau_\eps^U=\tauh+\tau_\eps(\delta)\}\to 1,\quad \eps\to 0,
\]
so that the asymptotics for $\tau_\eps^U$ is defined by that of $\tauh$ and $\tau_\eps(\delta)$. It is also clear that
one can set $\psi'_\eps=\eta^+_\eps$,
and $\phi'_\eps=Dg(\sgn(\eta^+_\eps)\delta e_1)Y_\eps(\tau_\eps)$, so that the limiting distribution of $\Theta'_\eps$ is given by
\begin{equation}
 \left(\eta^+_0,\ Dg(\sgn(\eta^+_0)\delta e_1)(\theta e_2),\ \frac{1}{\lambda_+}\log\frac{\delta}{|\eta_0^+|}\right),
\label{eq:limiting_distr_for_Theta_prime}
\end{equation}
where random variables $\eta_0^+$ and $\theta$ are defined in~\eqref{eq:eta_0_+} and~\eqref{eq:theta}

\section{Proof of Lemma~\ref{prop: eta_convergence} }\label{sec:close_to_saddle}

In this section we shall prove  Lemma~\ref{prop: eta_convergence} using several auxiliary lemmas. We start with some terminology.

\begin{definition}\rm
\label{def: Op} Given a family $(\xi_{\epsilon })_{\epsilon >0}$ of random 
variables or random vectors and a function $h:(0,\infty )\rightarrow (0,\infty )$ we say that $%
\xi_{\epsilon }=O_{\mathbf{p}}(h(\epsilon ))$ if for some $\eps_0>0$
distributions of $\left( \xi_{\epsilon}/h(\epsilon ) \right )_{0< \epsilon < \eps_0}$, form a tight family,
i.e.,\ for
any $\delta>0$ there is a constant $K_\delta>0$ such that 
\begin{equation*}
\mathbf{P}\left\{ |\xi_{\epsilon }|>K_{\delta }h(\epsilon )\right\} <\delta ,\quad 0<\epsilon <\epsilon _{0}.
\end{equation*}
\end{definition}

\begin{definition} \rm
A family of random variables or random vectors~$(\xi_\eps)_{\eps>0}$ is called slowly growing as $\eps\to 0$ (or just slowly growing) if
$\xi_\eps=O_{\Pp}(\eps^{-r})$ for all $r>0$. 
\end{definition}

Our first lemma estimates the martingale component of the solution of SDEs~\eqref{eq:SDE_changed_coord1} and~\eqref{eq:SDE_changed_coord2}.
 Let us define
\begin{eqnarray*}
S_{\epsilon }^{+}(T) &=&\sup_{t\leq  T }\left\vert
\int_{0}^{t}e^{-\lambda _{+}s}\sgm _{1}(Y_{\epsilon}(s))dW(s)\right\vert,\quad T>0, \\
S_{\epsilon }^{-}(T) &=&\sup_{t\leq T }\left\vert
\int_{0}^{t}e^{-\lambda _{-}(t-s)}\sgm_{2}(Y_{\epsilon}(s))dW(s)\right\vert,\quad T>0.
\end{eqnarray*}%

\begin{lemma}
\label{lemma: stoch_est}
Suppose $(\tau_\eps)_{\eps>0}$ is a family of stopping times (w.r.t.\ the natural
filtration of $W$). Then  
\[
S_{\epsilon }^{+}(\tau_\eps )=O_{\mathbf{P}}(1).
\]
If additionally $(\tau_\eps)_{\eps>0}$ is slowly growing, then $S_{\epsilon }^{-}(\tau_\eps )$ is
also slowly growing.
\end{lemma}

\begin{proof}
Both estimates are elementary. The first one is an easy consequence of the
martingale property of the stochastic integral involved in the definition of~$S_\eps^+$, 
and the BDG inequality (see~\cite[Theorem 3.3.28]{Karatzas--Shreve}). As for the second one,
we notice that the stochastic integral in the definition of $S^-_\eps$ behaves essentially like an Ornstein--Uhlenbeck process, 
and similar bounds apply.
\end{proof}

\begin{lemma}
\label{thm: estimate_non_linear} Suppose $Y_\eps$ is the solution of equations
 (\ref{eq:SDE_changed_coord1})--(\ref{eq:SDE_changed_coord2}) with initial conditions given by%
\begin{equation}
Y_{\epsilon,1 }(0)=\epsilon ^{\alpha }\chi _{\epsilon,1 }\quad \text{\rm  and }\quad Y_{\epsilon,2}(0)
=y_{2}+\epsilon ^{\alpha }\chi _{\epsilon,2 },  \label{eqn: Initial}
\end{equation}%
where distributions of random variables $(\chi _{\epsilon,1 })_{\epsilon >0}$ and $(\chi _{\epsilon,2
})_{\epsilon >0}$ form tight families. 
Let us fix any $R>0$ and denote $l_\eps= \tau_\eps^U\wedge(- \frac{\alpha}{\lambda_+}\log\eps + R)$ for $\eps>0$. Then
\[
\sup_{t\leq l_\eps}e^{-\lambda
t}|Y_{\epsilon,1 }(t)| = O_\Pp (\eps ^\alpha),
\]  and the family
\[
\left ( \eps^{-\alpha }\sup_{t\leq l_\eps}|Y_{\epsilon,2 }(t)-e^{-\lambda t }(y_{2}+\epsilon ^{\alpha }\chi
_{\epsilon,2})| \right )_{\eps>0}
\] is slowly growing.
\end{lemma}

\begin{proof} The tightness property implies that without loss of generality we can assume that $|\chi_{\epsilon,1 }|,|\chi_{\epsilon,2 }|<C$ for some constant $C>0$ and every $\epsilon >0$.

Let us fix $\gamma >0$. We can use Lemma~\ref{lemma: stoch_est} to take $c=c(\gamma
/3)>0$ such that 
\[
\mathbf{P}\{S_{\epsilon }^{+ }(l_\eps)>c\}<\gamma /2,
\]
and 
\[
\Pp\{ S_\eps^-(l_\eps)>c\eps^{-q}\}<\gamma/2,
\]
where $q$ is an arbitrary number satisfying $0<q<\alpha$. Let us introduce a constant $K=(3c)\vee C$ and
stopping times 
\begin{align*}
\beta _{+} &=\inf \left\{ t\ge 0:e^{-\lambda _{+}t}|Y_{\epsilon,1 }(t)|\geq
2K\epsilon ^{\alpha }\right\} , \\
\beta _{-} &=\inf \left\{ t\ge 0:|Y_{\epsilon,2 }(t)-e^{-\lambda
_{-}t}(y_2+\epsilon ^{\alpha }\chi _{\epsilon,2 })|\geq 2K\epsilon ^{\alpha - q}\right\} , \\
\beta&=\beta _{+}\wedge \beta_- \wedge l_\eps.
\end{align*}
We start with an estimate for $Y_{\epsilon,1 }$. Duhamel's principle 
for \eqref{eq:SDE_changed_coord1}, Theorem~\ref{Lemma: Def_H1&H2} and Lemma~\ref{lemma: stoch_est} imply that the estimate%
\begin{align}
\sup_{t\leq \beta }e^{-\lambda _{+}t}|Y_{\epsilon,1 }(t)| &\leq \epsilon
^{\alpha }K+K_{1}\int_{0}^{\beta }e^{-\lambda _{+}s}Y_{\epsilon,1
}(s)^{2}|Y_{\epsilon,2 }(s)|ds+K_{2}\frac{\epsilon ^{2}}{\lambda _{+}} +\epsilon S_{\epsilon }^{+}(\beta )  \notag \\
&\leq \epsilon ^{\alpha }K+K_{1}\int_{0}^{\beta }e^{-\lambda
_{+}s}Y_{\epsilon,1 }(s)^{2}|Y_{\epsilon,2 }(s)|ds+K_{2}\frac{\epsilon ^{2}}{\lambda _{+}}  +\epsilon \frac{K}{3} \label{eqn: x_k_espression}
\end{align}%
holds with probability at least $1-\gamma /2$. We analyze each term in the
RHS of equation \eqref{eqn: x_k_espression}.

Let us start with the integral in \eqref{eqn: x_k_espression}. For $%
s\leq \beta $, we see that 
\begin{align}
Y_{\epsilon,1 }(s)^{2}|Y_{\epsilon,2 }(s)| &\leq 4K^{2}\epsilon ^{2\alpha
}e^{2\lambda _{+}s}\left( |Y_{\epsilon,2 }(s)-e^{-\lambda
_{-}s}(y_2+\epsilon ^{\alpha }\chi _{\epsilon,2 })|+e^{-\lambda
_{-}s}|y_2+\epsilon ^{\alpha }\chi _{\epsilon ,2}|\right)   \notag \\
&\leq 8K^{3}\epsilon ^{3\alpha - q } e^{2\lambda _{+}s}+4K^{2}\epsilon ^{2\alpha
}e^{(2\lambda _{+}-\lambda _{-})s}(|y_2|+\epsilon ^{\alpha }C).
\notag 
\end{align}
Therefore,
\begin{align} \notag
K_{1}\int_{0}^{\beta }e^{-\lambda _{+}s}Y_{\epsilon,1 }(s)^{2} |Y_{\epsilon,1 }(s)|ds 
&\leq \frac{8K^{3}K_{1}e^{\lambda_+R}}{\lambda _{+}}  \epsilon ^{2\alpha - q } \\ 
&+4K_{1}K^{2}  \epsilon ^{2\alpha }(|y_{2}|+\epsilon ^{\alpha} C)\int_0^\beta e^{(\lambda_+ - \lambda _{-}) s} ds  \notag \\
&\leq K\epsilon ^{\alpha }/12 +  5K_{1}K^{2}  \epsilon ^{2\alpha }|y_{2}| \int_0^\beta e^{(\lambda_+ - \lambda _{-}) s} ds \label{eqn; bound_x} 
\end{align}
for all $\eps >0$ small enough. Notice that this is a rough estimate, the constants on the r.h.s.\ are not optimal but sufficient for our purposes. This also applies to some
other estimates in this proof.

Let us estimate the integral on the r.h.s.\  of \eqref{eqn; bound_x}. When $\lambda_+ > \lambda_-$, the integral is bounded by
\[
\frac{1}{\lambda_+ - \lambda_-} e^{ (\lambda_+ - \lambda_- ) \beta} \leq 
\frac{e^{(\lambda_+ - \lambda_-)R}}{\lambda_+ - \lambda_-} \eps^{ - \alpha+\alpha \lambda_- / \lambda_+};
\]
if $\lambda_+ < \lambda_-$, then the integral on the r.h.s of~\eqref{eqn; bound_x}
is bounded by~$(\lambda_- - \lambda_+)^{-1}$; if $\lambda_+=\lambda_-$, then the integral is bounded by 
$2\alpha\lambda_+^{-1}|\log \eps|$. Hence, for some constant $K_{\lambda_+,\lambda_-}>0$ and $\epsilon >0$ small enough,
\begin{align} 
\notag
K_{1}\int_{0}^{\beta }e^{-\lambda _{+}s}Y_{\epsilon,1 }(s)^{2}|Y_{\epsilon,2 }(s)|ds 
&\leq K\epsilon ^{\alpha }/12 + K_{\lambda_+,\lambda_-} \eps^{2\alpha -\alpha(1-\lambda_-/\lambda_+)^+} |\log\eps| \notag \\
&\leq K \eps^\alpha /6 \label{eqn: ineq_1}.
\end{align}%
Also, for $\epsilon >0$ small enough, 
\begin{equation}
K_{2}\epsilon ^{2}/\lambda _{+}+\eps K/3 <K\epsilon ^{\alpha }/2.  \label{eqn: ineq_2}
\end{equation}%
From~\eqref{eqn: x_k_espression}, \eqref{eqn: ineq_1} and \eqref{eqn: ineq_2} we get that for all $\eps>0$ small enough, the event   
\begin{equation*}
A=\left\{ \sup_{t\leq \beta }e^{-\lambda _{+}t}|Y_{\epsilon,1 }(t)|\leq
5K\epsilon ^{\alpha }/3\right\} 
\end{equation*}%
is such that $\Pp (A)>1-\gamma /2$.

Let us now consider $Y_{\epsilon,2 }(t)$ and denote 
\[
Z_{\epsilon }(t)=Y_{\epsilon,2 }(t)-e^{-\lambda_{-}t}(y_{2}+\epsilon ^{\alpha }\chi _{\epsilon,2 }).
\]
Duhamel's principle for $Y_{\eps,2}$, the definition of $\beta $, Theorem~\ref{Lemma: Def_H1&H2} and  Lemma~\ref{lemma: stoch_est}
imply that the inequalities
\begin{align}
\notag
\sup_{t\leq \beta }|Z_{\epsilon }(t)| &\leq K_1 \sup_{t \leq \beta}\int_{0}^{t}e^{-\lambda _{-}(t-s)}|Y_ {\eps,1}(s)|^\aone |Y_{\eps,2}(s)|^\atwo ds+K_2 \eps^{2} / \lambda_-  
+\epsilon S_{\epsilon }^{-}(\beta )  \\
&\leq K_1 \sup_{t \leq \beta}\int_{0}^{t}e^{-\lambda _{-}(t-s)}|Y_ {\eps,1}(s)|^\aone |Y_{\eps,2}(s)|^\atwo ds \notag \\
& \hspace{1.5 in} + \eps^{\alpha-q} \left(  K_2 \eps^{2-\alpha+q} / \lambda_- +\epsilon^{1-\alpha+q} S_{\epsilon }^{-}(\beta ) \right)
\notag \\
&\leq 2^\aone \eps^{\alpha \aone} K^\aone K_1 \sup_{t \leq \beta} e^{-\lambda_- t} \int_{0}^{t} e^{(\lambda _{-} + \aone \lambda_+ )s}|Y_{\eps,2}(s)|^\atwo ds + \eps^{\alpha-q}K/2  \label{eqn: y_k_espression}
\end{align}%
hold with probability at least $1-\gamma /2$ and for all $\epsilon >0$ small enough.
We analyze the integral term in \eqref{eqn: y_k_espression}. Note that, from the definition of $\beta$, and the inequality $(a+b)^r \leq 2^{r-1} ( a^r + b^r )$ we have that for any $t \leq \beta$ and any $\eps >0$ small enough,
\begin{align*}
|Y_{\eps,2}(t)|^\atwo &\leq 2^{\atwo-1} Z_\eps(t)^\atwo + 2^{\atwo -1 }e^{-\atwo \lambda_- t}|y_2 + \eps^\alpha \chi_{\eps,2}|^\atwo \\
& \leq 2^{2\atwo-1} K^\atwo \eps^{ (\alpha -q)\atwo } + 2^{2( \atwo -1)}e^{-\atwo \lambda_- t}|y_2|^\atwo \\
& \hspace{2.255 in}+2^{2( \atwo -1)}\eps^{\alpha \atwo }e^{-\atwo \lambda_- t}|\chi_{\eps,2}|^\atwo \\ 
&\leq \eps^{\atwo ( \alpha - q)} 2^{2 (\atwo -1)} \left(  2K^\atwo + \eps^{q \atwo} |\chi_{\eps,2}|^\atwo \right) + 2^{2(\atwo-1)} e^{- \atwo \lambda_- t} |y_2|^\atwo.
\end{align*}
Hence there is a constant $K_\alpha>0$ such that
\[
|Y_{\eps,2}(t) | ^\atwo \leq \eps^{\atwo ( \alpha - q)} K_\alpha + K_\alpha e^{-\atwo \lambda_- t},\quad t\leq \beta.
\] 
Using the last inequality, the definition of $\beta$, and the fact  $\aone \lambda _+ - (\atwo - 1)\lambda_-=0$ from Theorem~\ref{Lemma: Def_H1&H2}, we get
\begin{align}
\notag
\eps^{\alpha \aone} e^{-\lambda_- t} & \int_{0}^{t} e^{(\lambda _{-} + \aone \lambda_+ )s}| Y_{\eps,2}(s)|^\aone ds \\
\notag &\leq \eps^{\alpha( \aone + \atwo)} e^{\lambda_+ \aone \beta} \frac{K_\alpha \eps^{-q\atwo}}{\lambda_- + \aone \lambda_+}
 + K_\alpha \eps^{\alpha \aone }\int_{0}^{t} e^{(\aone \lambda _+ - (\atwo - 1)\lambda_- )s}ds \\
& \leq \eps^{ (\alpha - q) \atwo} \frac{K_\alpha e^{\lambda_+\alpha_1^- R}}{\lambda_- + \aone \lambda_+} + K_\alpha \eps^{\alpha \aone } \beta.
\label{eqn: integral_y_1}
\end{align}
Again, from Theorem~\ref{Lemma: Def_H1&H2} we know that $\aone \geq 1$ and $\atwo \geq 2$ which together with~\eqref{eqn: integral_y_1} imply that for all $\eps >0$ small enough
\begin{equation}
2^\aone \eps^{\alpha \aone} K^\aone K_1 \sup_{t \leq \beta} e^{-\lambda_- t} \int_{0}^{t} e^{(\lambda _{-} + \aone \lambda_+ )s}|Y_{\eps,2}(s)|^\atwo ds \leq K \eps^{\alpha - q} /6.
\label{eqn: y_integral_final}
\end{equation}
Using~\eqref{eqn: y_integral_final} and~\eqref{eqn: y_k_espression} we conclude that
the event 
\begin{equation*}
B=\left\{ \sup_{t\leq \beta }|Y_{\epsilon,2 }(t)-e^{-\lambda
_{-}t}(y_{2}+\epsilon ^{\alpha }\chi _{\epsilon,2 })|\leq 2K\epsilon ^{\alpha-q
}/3\right\} 
\end{equation*}%
is such that $\mathbf{P}(B) \geq 1-\gamma /2,$ for all $\epsilon >0$ small
enough.

The proof will be complete once we show that $\beta =l_\eps$ with probability at least $1-\gamma $. The latter is a consequence of the following
chain of inequalities that hold for all $\eps>0$ small enough: 
\begin{align*}
\mathbf{P}\{\beta _{+}\wedge \beta _{-} \leq l_\eps \}&\leq \mathbf{P}\left( \{\beta _{+}\wedge \beta _{-}\leq l_\eps\} \cap A\cap B \right)+\mathbf{P}(A^{c})+\mathbf{P}(B^{c}) \\
&\leq \mathbf{P}\left( \{\beta _{+}\wedge \beta _{-}\leq l_\eps \} \cap A\cap B\right)+\gamma \\
&\leq \mathbf{P}\left(  \{\beta _{+}\leq \beta _{-}\wedge l_\eps\} \cap A\right)+\mathbf{P}\left( \{\beta _{-}\leq \beta _{+}\wedge l_\eps\} \cap B \right)+\gamma \\
&=\mathbf{P}\{2\leq 5/3\}+\mathbf{P}\{2\leq 2/3\}+\gamma =\gamma .
\end{align*} 
\end{proof}

\medskip

Let us now analyze the evolution of the process $Y_\eps$ up to time $\tauh \wedge \tau_\eps^U$. We start with an application of Duhamel's principle: 
\begin{align}
Y_{\epsilon,1 }(t) &=e^{\lambda _{+}t}Y_{\epsilon,1 }(0)+\int_{0}^{t}e^{\lambda
_{+}(t-s)}H_{1}(Y_{\epsilon }(s),\eps)ds+\epsilon e^{\lambda _{+}t}%
{N}_{\epsilon }^{+}(t),  \label{eqn: x_duhamel} \\
Y_{\epsilon,2 }(t) &=e^{-\lambda _{-}t}Y_{\epsilon,2}(0)+\int_{0}^{t}e^{-\lambda _{-}(t-s)} H_{2}(Y_{\epsilon
}(s),\eps)ds+\epsilon {N}_{\epsilon }^{-}(t),  \label{eqn: y_duhamel}
\end{align}%
where ${N}_{\epsilon }^{\pm}(t)$ are defined by
\begin{align}
{N}_{\epsilon }^{+}(t) &=\int_{0}^{t}e^{-\lambda _{+}s}\sgm
_{1}(Y_{\epsilon }(s))dW(s),\notag \\
{N}_{\epsilon }^{-}(t) &=\int_{0}^{t}e^{-\lambda _{-}(t-s)}\sgm
_{2}(Y_{\epsilon }(s))dW(s). \label{eqn: def_N-}
\end{align}

\begin{lemma}
\label{lemma: y_convergence}
\[
\sup_{t \leq \hat{\tau_\eps} }
|Y_{\eps,2}(t) - e^{ -\lambda_- t} y_2|=O_{\Pp}(\eps^{\alpha p}).
\] 
\end{lemma}
\begin{proof}
Duhamel's princinple, Theorem~\ref{Lemma: Def_H1&H2},  and the definition of~$\tauh$
imply that for some $K>0$,
\begin{align*}
|Y_{\eps,2} (t) - e^{ -\lambda_- t} y_2| &\leq \eps^\alpha |\chi_{\eps,2}| +\int_0^t e^{- \lambda_- (t - s)} \left(K_1|Y_{\eps,1} (s)|Y_{\eps,2}^2 (s) + K_2\eps^2 \right) ds + \eps S_\eps^- (t)\\
&\leq \eps^\alpha |\chi_{\eps,2}| + K \eps^{ \alpha p} + \eps^{\alpha p} \left( \eps^{1-\alpha p } S_\eps^- (\tauh) \right)
\end{align*}
for any $t \in (0, \tauh )$. The result follows since by Lemma~\ref{lemma: stoch_est}
the r.h.s.\ is $O_\Pp (\eps^{\alpha p})$
\end{proof}

As a simple corollary of this lemma, the first statement in Theorem~\ref{prop: eta_convergence} follows:
\begin{corollary} As ${\eps \to 0}$,
\label{cor: tau_eps<infty}
\[
 \Pp \{  \tau_\eps^U < \tauh \}\to 0.
\]In particular,~\eqref{eq:exit_at_tauh_along_axis_1}  holds true.
\end{corollary}

\begin{lemma}
\label{lemma: ito_convergence}
Let
\[
{N}_{0}^{+}(t) =\int_{0}^{t}e^{-\lambda _-s}\sgm_{1}(0,e^{-\lambda _{-}s}y_{2})dW.
\]
Then
\[
\sup_{t \leq \tauh} |{N}_\eps^+ (t) - {N}_0^+ (t) | \overset{L^2} {\longrightarrow}0,
 \quad \eps \to 0.
\]
\end{lemma}  
\begin{proof} BDG inequality implies that for some constants $C_1,C_2>0$,
\begin{align}\notag 
\mathbf{E}\sup_{t\leq \tauh}|{N}_{\epsilon
}^{+}(t)-{N}_{0}^{+}(t)|^{2} &\leq C_{1}\mathbf{E}\int_{0}^{\hat%
\tau _{\epsilon }}e^{-2\lambda _{+}s}|\sgm_1 (Y_{\epsilon,1}(s),Y_{\eps,2}(s))- (0,e^{-\lambda _{-}s}y_{2})|^{2}ds \\
&\leq C_2\mathbf{E}\sup_{t\leq \hat{\tau} _{\epsilon }}|\sgm_1
(Y_{\epsilon,1 }(s),Y_{\epsilon,2 }(s))-\sgm_1 (0,e^{-\lambda _{-}s}y_{2})|^{2}.
\label{eq:Gaussian_approx}
\end{align}%
From Lemma~\ref{lemma: y_convergence} and the definition 
of~$\tauh$, it follows that 
\begin{equation}
\sup_{t\leq\hat\tau_\eps}\left |(Y_{\eps,1}(t),Y_{\eps,2}(t))-(0,e^{-\lambda_- t}y_2)\right|=O_{\Pp}(\eps^{\alpha p} ).
\label{eq:approx_along_stable_manifold}
\end{equation}
The desired convergence follows now from \eqref{eq:Gaussian_approx}, 
\eqref{eq:approx_along_stable_manifold}, and the boundedness and Lipschitzness of $\sgm_1$.
\end{proof}

We are now in position to give the first rough asymptotics for the time~$\tauh$. From now on we restrict ourselves to the 
 event $\{\tau_\eps^U > \tauh\}$ since due to Corollary~\ref{cor: tau_eps<infty} its probability is arbitrarily high.
\begin{lemma}
\label{prop: tau_bar_convergence} As $\eps\to0$,
\[
\Pp  \left \{  \tauh > -\frac{\alpha}{\lambda_+} \log \eps \right \} \to 0.
\]
\end{lemma}
\begin{proof}
Let $u_\eps$ be the solution to the following SDE:
\begin{align*}
du_\eps(t) &= \lambda_+ u_\eps(t)dt + \eps \sgm_1 (Y_\eps(t))dW(t),\\
u_\eps(0)&=\eps^\alpha \chi_{\eps,1}.
\end{align*}
Let us take $\delta_0\in(0,1)$ to be specified later and consider the following stopping time 
\[
\tauw=\inf \left \{ t: |u_\eps(t)|= \eps^{\alpha \delta_0} \right \}.
\] 
Duhamel's principle for $u_\eps$ writes as 
\begin{align*}
u_\eps(t) &= \eps^\alpha e^{\lambda_+ t}\chi_{\eps,1} + \eps e^{\lambda_+ t} {N}_\eps^+(t) \\
&= \eps^\alpha e^{\lambda_+ t } \wt{\eta}_\eps (t),
\end{align*}
with 
\begin{equation}
\wt{\eta}_\eps (t)= \chi_{\eps,1} + \eps^{1-\alpha} {N}_\eps^+ (t).
\label{eq:eta-tilde} 
\end{equation}
Hence, the definition of $\tauw$ implies $\eps^{\alpha \delta_0} = \eps^{\alpha} e^{\lambda_+ \tauw} |\wt{\eta}_\eps (\tauw)|$, so that 
\[
\tauw = -\frac{\alpha}{\lambda_+} (1-\delta_0 ) \log\eps -\frac{1}{\lambda_+} \log |\wt{\eta}_\eps (\tauw)|.
\]
Due to~\eqref{eq:eta-tilde} and Lemma~\ref{lemma: ito_convergence}, the distributions of $\frac{1}{\lambda_+} \log |\wt{\eta}_\eps (\tauw)|$ form a tight family. Therefore,
\begin{equation}
\label{eq:tauh_grows}
\lim_{\eps \to 0} \Pp \left \{  \tauw > -(1-\delta_0^2)\frac{\alpha}{\lambda_+} \log \eps \right \}=0.
\end{equation}
This fact allows us to use Lemma~\ref{thm: estimate_non_linear} to estimate $Y_\eps$
up to $\tauh \wedge \tauw$.
From~\eqref{eqn: x_duhamel}, the difference $\Delta_\eps=Y_{\eps,1}-u_\eps$ is given by
\[
\Delta_\eps (t) = e^{\lambda_+ t} \int_0^t e^{ -\lambda_+ s}  H_1(Y_\eps(s),\eps ) ds. 
\]
We can use~\eqref{eq:tauh_grows} to justify the application of Lemma~\ref{thm: estimate_non_linear} up to time $\tauh \wedge \tauw$. Then, we combine Theorem~\ref{Lemma: Def_H1&H2}, Lemma~\ref{thm: estimate_non_linear}, and the definition of $\tauh$ to see that 
\begin{align*}
\sup_{t\leq \tauh \wedge \tauw} e^{- \lambda_+ t} |H_1(Y_{\epsilon}(t), \eps)|&\leq K_1\sup_{t\leq \tauh \wedge \tauw} \left( \left(e^{-\lambda_+ t} |Y_{\eps,1}(t)|\right) |Y_{\eps,1}(t)|\cdot|Y_{\eps,2}(t)|\right) + K_2 \eps^2 \\
&=O_\Pp \left( \eps^{\alpha + \alpha p }  \right)\\
\end{align*}
and 
\[
e^{\lambda_+ \tauh \wedge \tauw } = O_\Pp \left ( \eps^{ -\alpha (1- \delta_0^2)  } \right).
\]
These two estimates together with~\eqref{eq:tauh_grows} imply 
\[
\sup_{t\leq \tauh \wedge \tauw } | \Delta_\eps (t) | = O_\Pp \left (  \eps^{ \alpha (p+\delta_0^2)  } |\log\eps|\right).
\]
On one hand,~\eqref{eq:tauh_grows} implies
\[ 
\Pp\left(\left \{  \tauh > -\frac{\alpha}{\lambda_+} \log 
\eps \right \} \cap \{\tauh \leq \tauw \}\right)\to 0.
\] 
On the other hand, if  $ \tauh > \tauw$ then
\[
|Y_{\eps,1} (\tauw)| = \left| \eps^{ \alpha \delta_0 } + O_\Pp ( \eps ^{ \alpha ( p + \delta _0^2 )} |\log\eps|) \right|,
\]
and
\[
 |Y_{\eps,1}(\tauw)|< \eps^{\alpha p}. 
\]
These relations contradict each other for sufficiently small $\eps$ if we choose $\delta_0 < p$.
So, this choice of $\delta_0$ guarantees that
 $\Pp \left \{  \tauh > \tauw \right \} \to 0$ implying the result. 
\end{proof}

\bigskip

\begin{proof}[Proof of Lemma~\ref{prop: eta_convergence}] 
Recall that we work on the high probability event $\{  \tauh < \tau_\eps^U \}$.  Hence, for each $\epsilon >0$, we have
the identity 
\begin{equation*}
\epsilon ^{\alpha p}=\epsilon ^{\alpha }e^{\lambda _{+}\hat{\tau}_{\epsilon }%
}|\eta _{\epsilon }^{+}|.
\end{equation*}%
Solving for $\tauh$ and then plugging it back into $%
Y_{\epsilon,1 }$, we get%
\begin{align}
\hat{\tau}_{\epsilon } &=-\frac{\alpha }{\lambda _{+}}(1-p)\log \epsilon -%
\frac{1}{\lambda _{+}}\log |\eta _{\epsilon }^{+}|,
\label{eqn: tau_explicit} \\
Y_{\epsilon,1 }(\hat\tau _{\epsilon }) &=\epsilon ^{\alpha p} \sgn(\eta
_{\epsilon }^{+}).  \notag
\end{align}
Using this information we are in position to get the asymptotic behavior of the random variables $\eta
_{\epsilon }^{\pm }$. First, from relation \eqref{eqn: x_duhamel} we get
\begin{equation}
\eta _{\epsilon }^{+}=\chi_{\epsilon,1 }+\epsilon ^{-\alpha }\int_{0}^{
{\tauh}}e^{-\lambda _{+}s}H_{1}(Y_{\epsilon
}(s),\eps)ds+\epsilon ^{1-\alpha }{N}_{\epsilon }^{+}(\hat{\tau _{\epsilon }}%
). \label{eqn: eta+}
\end{equation}
Using \eqref{eqn: tau_explicit} in \eqref{eqn: y_duhamel} we get%
\begin{align}
\eta _{\epsilon }^{-} &=|\eta _{\epsilon }^{+}|^{\lambda _{-}/\lambda
_{+}}(y_{2}+\epsilon ^{\alpha }\chi _{\epsilon,2 })+|\eta _{\epsilon
}^{+}|^{\lambda _{-}/\lambda _{+}}\int_{0}^{\hat{\tau} _{\epsilon }%
}e^{\lambda _{-}s} H_{2}(Y_{\epsilon }(s),\eps)ds  \notag \\
&+\epsilon ^{1-\alpha (1-p)\lambda _{-}/\lambda _{+}}{N}_{\epsilon
}^{-}(\hat{\tau} _{\epsilon }).  \label{eqn: eta-}
\end{align}%

The main part of the proof is based on representations \eqref{eqn: tau_explicit}--\eqref{eqn: eta-}.

Lemma~\ref{prop: tau_bar_convergence} allows us to use the estimates established in
Lemma~\ref{thm: estimate_non_linear} up to time~$\tauh$. In particular, now we can
conclude that the family
\begin{equation}
\left( \eps^{-\alpha} \sup_{t \leq \tauh} |Y_{\eps,2} (t) - e^{-\lambda_- t} y_2| \right)_{\eps>0}
\label{eq:y_slowly_growing} 
\end{equation}
is slowly growing thus improving Lemma~\ref{lemma: y_convergence}.

To obtain the desired convergence for $\eta_\eps^+$, we analyze the r.h.s.\ of~\eqref{eqn: eta+}
term by term. The covergence of the first term was one of our assumptions. For the second one,
we need to estimate $H_1(Y_\eps,\eps)$. 
 Using Lemma~\ref{thm: estimate_non_linear}, the boundness of~$Y_{\eps,2}$ and the definition of  
$\tauh$, we see that
\begin{equation}
\sup_{t\leq \hat{\tau}_\eps} e^{-\lambda_+ t} Y_{\eps,1}^2 (t) |Y_{\eps,2} (t)| = O_\Pp (\eps^{\alpha+\alpha p} ). \label{eqn: eta+_stoch}
\end{equation}
This estimate and Theorem~\ref{Lemma: Def_H1&H2} imply that
\begin{align}
\eps^{-\alpha} \int_0^{\tauh} e^{-\lambda_+ s}
H_1(Y_\eps(s),\eps)ds &\leq K_1 \eps^{-\alpha} \int_0^{\tauh} e^{-\lambda_+ s} Y_{\eps,1}^2 (s) |Y_{\eps,2}(s)|ds 
 + \frac{K_2}{\lambda_+} \eps^{2-\alpha} \notag \\
& =O_\Pp ( \eps^{\alpha p} |\log\eps| ).\notag
\end{align}
Let us estimate the third term in \eqref{eqn: eta+}. We can use the last estimate along with \eqref{eqn: eta+} and Lemma~\ref{lemma: ito_convergence} to conclude that 
the distributions of positive part of $\lambda_+^{-1} \log | \eta_\eps^+|$ form a tight family.
Therefore, \eqref{eqn: tau_explicit} implies that 
\[
\tauh  \overset{\Pp}{\to} \infty,\quad \eps \to 0. 
\]
Combined with It\^o isometry  and Lemma~\ref{lemma: ito_convergence}, this implies
\[
{N}_{\epsilon }^{+}(\hat\tau _{\epsilon })\overset{L^{2}}{%
\longrightarrow }N^+,\quad\epsilon \rightarrow 0,
\]
which completes the analysis of $\eta_\eps^+$ and, due to~\eqref{eqn: tau_explicit}, of $\tauh$.

\medskip

To obtain the convergence of $\eta_\eps^-$, we study~\eqref{eqn: eta-}.  
Combining \eqref{eq:y_slowly_growing}, the inequality 
\[
|Y_{\eps,1}(t)| Y_{\eps,2}^2 (t) \leq 2 |Y_{\eps,1} (t)| \left( |Y_{\eps,2}(t)-e^{-\lambda_- t}y_2|^2 + e^{ -2 \lambda_- t} y_2^2 \right ),
\] and the definition of $\tauh$ we see that for any $q\in(0,\alpha p)$,
\[
\sup_{t \leq \hat\tau_\eps }e^{\lambda_- t } |Y_{\eps,1}(t)| Y_{\epsilon,2 }^2(t) =O_\Pp \left( \eps^{\alpha p + \alpha - q} e^{\lambda_- \tauh }+\epsilon ^{\alpha p } \right ).
\]
Hence, as a consequence of Theorem~\ref{Lemma: Def_H1&H2} and~\eqref{eqn: tau_explicit} we have  
\begin{align*}
\int_{0}^{\tauh}e^{\lambda _{-}s}  H_{2}(Y_{\epsilon }(s),\eps)ds &=O_\Pp \left(\left(\epsilon ^{\alpha p -q + \alpha }e^{\lambda _{-}\tauh} + \epsilon ^{\alpha p }\right)|\log \epsilon |\right) \\
&=O_{\mathbf{P}} \left(\left(\epsilon ^{\alpha (1-(1-p)\lambda _{-}/\lambda _{+}) + (\alpha p - q)} + \epsilon ^{\alpha p}\right)|\log \epsilon | \right).
\end{align*}%
Combining this and Lemma~\ref{lemma: stoch_est} in \eqref{eqn: eta-} we obtain
\begin{align*}
\eta _{\epsilon }^{-} &=|\eta _{\epsilon }^{+}|^{\lambda _{-}/\lambda
_{+}}y_2+O_{\mathbf{P}}(\eps^\alpha)+O_{\mathbf{P}} \left(\left(\epsilon ^{\alpha (1-(1-p)\lambda _{-}/\lambda _{+}) + (\alpha p - q)}+\epsilon ^{\alpha p}\right)|\log \epsilon | \right) \\
& + O_{\mathbf{P}}\left(\epsilon ^{1-\alpha (1-p)\lambda _{-}/\lambda _{+} - q}\right)
\end{align*}
which finishes the proof of Lemma~\ref{prop: eta_convergence} by choosing $q$ small enough.
\end{proof}

\section{Proof of Lemma~\ref{Thm:after_tauh}}\label{sec:unstable_manifold}

Consider the solution to system~\eqref{eq:SDE_changed_coord1}--\eqref{eq:SDE_changed_coord2}
equipped with initial conditions~\eqref{eq:restart_at_tauh} satisfying~\eqref{eq:condition_for_theorem_along_unstable}.
Let us restrict the analysis to the arbitrary high probability event 
\[
\{ |\eta _{\epsilon }^{\pm }|\leq K_{\pm } \},
\]
for some constants $K_\pm>0$.

\begin{lemma}
\label{lemma: y_after_global} Let $p\in (0,1)$ satisfy~\eqref{eqn: p_prop}, and
let $(t_\eps)_{\eps>0} $ be a slowly growing family of stopping times. Consider $t_\eps'=t_\eps \wedge \tau_\eps^U$, then for any $\gamma >0$,
\begin{equation*}
\lim_{\epsilon \rightarrow 0}\mathbf{P}\left\{\sup_{t\leq t_\eps' }|Y_{\epsilon,2
}(t)|\leq (K_{-}+\gamma )\epsilon ^{\alpha (1-p)\lambda _{-}/\lambda
_{+}}\right\}=1.
\end{equation*}
\end{lemma}

\begin{proof}
Let $\gamma >0$. We recall that  $N_\eps^-$ is defined in~\eqref{eqn: def_N-} and introduce the process
\begin{equation}
M_{\epsilon }(t) ={N}_{\epsilon }^{-}(t)+\epsilon
\int_{0}^{t}e^{-\lambda_- (t-s)}\Psi _{2}(Y_{\epsilon }(s))ds, \label{eqn: M_def}
\end{equation}
where $\Psi_2$ was introduced in Theorem~\ref{Lemma: Def_H1&H2},
and the stopping time 
\begin{equation*}
\beta _{\epsilon }=\inf \left\{t:|Y_{\epsilon,2 }(t)|>(K_{-}+\gamma )\epsilon
^{\alpha (1-p)\lambda _{-}/\lambda _{+}}\right\}.
\end{equation*}%
Using the fact that $Y_{\epsilon,1 }$ is bounded, it is easy to see that
there is a constant $%
K_{\lambda _{-}}$ independent of $t$, so that for any $t\leq \beta_\eps\wedge t_\eps'$,
we have
\begin{equation*}
\int_{0}^{t}e^{-\lambda _{-}(t-s)}|Y_{\epsilon,1 }(s)|Y_{\epsilon,2}^2(s)ds\leq K_{\lambda _{-}}\epsilon ^{ 2\alpha
(1-p)\lambda _{-}/\lambda _{+}}.
\end{equation*}%
This estimate, along with Duhamel's principle and Theorem~\ref{Lemma: Def_H1&H2} implies that for some constant $C>0$ and any $t \leq \beta_\epsilon \wedge t_\eps'$,%
\begin{align*}
|Y_{\epsilon,2 }(t)| &\leq \epsilon ^{\alpha (1-p)\lambda _{-}/\lambda
_{+}}|\eta^-_{\epsilon }|+K_1\int_{0}^{t}e^{-\lambda _{-}(t-s)}|Y_{\eps,1}(s)|Y_{\epsilon,2 }^2 (s)ds+\epsilon \sup_{t \leq \beta_\eps }|M_{\epsilon }(t)| \\
&\leq \epsilon ^{\alpha (1-p)\lambda _{-}/\lambda _{+}}K_{-}+C \epsilon ^{2\alpha (1-p)\lambda _{-}/\lambda _{+}}+\epsilon
\sup_{t\leq \beta _{\epsilon }}|M_{\epsilon }(t)|.
\end{align*}%
Hence, using Lemma \ref{lemma: stoch_est} to estimate $M_\eps$, we obtain that
\begin{align*}
\mathbf{P}\{\beta _{\epsilon } < t_\eps' \} &=\mathbf{P}\left\{\sup_{t \leq \beta_{\epsilon } \wedge t_\eps' } |Y_{\epsilon,2 }(t)| \geq (K_{-}+\gamma )\epsilon ^{\alpha
(1-p)\lambda _{-}/\lambda _{+}} \right\} \\
&\leq \mathbf{P}\left\{C\epsilon ^{\alpha (1-p)\lambda_{-}/ \lambda _{+}}+\epsilon ^{1-\alpha (1-p)\lambda _{-}/\lambda
_{+}}\sup_{t\leq \beta _{\epsilon }}|M_{\epsilon }(t)|\geq \gamma
 \right\}
\end{align*}%
converges to $0$ as $\eps \rightarrow 0$ proving the lemma.
\end{proof}

\begin{lemma}
\label{lemma: x_after_global}Under the assumptions of lemma \ref{lemma:
y_after_global}, for any $\rho \in (0,\frac{\alpha p}{%
\lambda _{+}}] $, $\gamma >0$, and $C>0$, define $\rho_\eps~=(-\rho \log \epsilon+C )\wedge \tau_\eps^U $. Then,  we have
\begin{equation*}
\lim_{\epsilon \rightarrow 0}\mathbf{P}\left\{\sup_{t\leq\rho_\eps}|Y_{\epsilon,1 }(t)|e^{-\lambda _{+}t}\leq (1+\gamma )\epsilon ^{\alpha
p}\right\}=1.
\end{equation*}
\end{lemma}

\begin{proof}
Define the stopping time 
\begin{equation*}
\beta _{\epsilon }=\inf \left\{t:|Y_{\epsilon,1 }(t)|e^{-\lambda _{+}t}\geq
(1+\gamma )\epsilon ^{\alpha p}\right\}.
\end{equation*}%
As a consequence of Duhamel's principle and Theorem~\ref{Lemma: Def_H1&H2} we get the bound 
\begin{align*}
\sup_{t\leq \beta _{\epsilon }\wedge \rho_\eps }  |Y_{\epsilon,1}(t)|e^{-\lambda _{+}t} 
\leq &\epsilon ^{\alpha p}+K_{1}\int_{0}^{\beta _{\epsilon }\wedge \rho_\eps}e^{-\lambda _{+}s}Y_{\epsilon,1}^{2}(s)|Y_{\epsilon,2}(s)|ds\\
&\quad +\epsilon ^{2}K_{2}\lambda _{+}^{-1} +\epsilon S_{\epsilon }^{+}(\beta_\eps ). 
\end{align*}
This estimate together with Lemma~\ref{lemma: y_after_global}, Lemma~~\ref{lemma: stoch_est} and the defintion of $\rho_\eps$ implies that for any small $\delta>0$ we can find a constant $K>0$, so that with probability bigger than $1-\delta$, the inequalities
\begin{align*}
\sup_{t\leq \beta _{\epsilon }\wedge \rho_\eps }  |Y_{\epsilon,1}(t)|e^{-\lambda _{+}t}
&\leq \epsilon ^{\alpha p}+K\epsilon ^{\alpha p+\alpha (1-p)\lambda
_{-}/\lambda _{+}}(\beta_\eps\wedge \rho_\eps)+K\eps \\
&\leq \epsilon ^{\alpha p}(1+2K\rho \epsilon ^{\alpha (1-p)\lambda
_{-}/\lambda _{+}}|\log \epsilon |+K \epsilon ^{1-\alpha p}),
\end{align*}
hold  for all $\epsilon >0$ small enough. 
Hence, for any small enough $\eps>0$,
\begin{align*}
\mathbf{P}\left\{\beta _{\epsilon } <\rho_\eps\right\} &=\mathbf{P}%
\left\{\sup_{t\leq \beta _{\epsilon }\wedge \rho_\eps}|Y_{\epsilon,1
}(t)|e^{-\lambda _{+}t}\geq (1+\gamma )\epsilon ^{\alpha p}\right\} \\
&\leq \mathbf{P}\left\{K\rho \epsilon ^{\alpha (1-p)\lambda _{-}/\lambda _{+}}|\log
\epsilon |+K\epsilon ^{1-\alpha p}\geq \gamma \right\} + \delta,
\end{align*}%
which implies the result.
\end{proof}

\medskip

The following is an important consequence of Lemma~\ref{lemma: y_after_global}:
\begin{corollary} With $\tau_\eps$ as in~\eqref{eqn: tau_def} it holds that
\[
\lim_{\eps \to 0} \Pp \{ \tau_\eps^U < \tau_\eps \}=0.
\]
In particular,~\eqref{eq:exit_through_unstabale} holds.
\end{corollary}

From now on, we restrict our analysis to the high probability event $\{ \tau_\eps^U \geq \tau_\eps \}$.

Let $\theta _{\epsilon }^{+}=\epsilon ^{-\alpha p}e^{-\lambda _{+} \tau _\eps}Y _{\epsilon,1 }(\tau _{\epsilon })$.
Then, \eqref{eqn: tau_def} implies  
\begin{equation}
\tau _{\epsilon } =-\frac{\alpha p}{\lambda _{+}}\log \epsilon +\frac{1}{%
\lambda _{+}}\log \frac{\delta }{|\theta _{\epsilon }^{+}|},
\label{eqn: sigma_equal} 
\end{equation}
and
\[
Y _{\epsilon,1 }(\tau_{\epsilon }) =\delta \sgn\theta _{\epsilon }^{+}.
\]
Our analysis of these expressions will be based on the next formula which directly follows 
from Duhamel's principle: 
\begin{equation}
\theta _{\epsilon }^{+}=\sgn\eta_{\epsilon }^{+}+\epsilon ^{-\alpha
p}\int_{0}^{\tau _{\epsilon }}e^{-\lambda _{+}s} H_{1}(Y_\eps(s),\eps)ds+\epsilon ^{1-\alpha p}N_{\epsilon
}^{+}(\tau _{\epsilon }).  \label{eqn: theta+}
\end{equation}%

The main term in the r.h.s.\ of~\eqref{eqn: theta+} is $\sgn\eta_{\epsilon }^{+}$. We need to estimate the other two terms. Lemma~\ref{lemma: stoch_est} implies that $\epsilon ^{1-\alpha p} N_{\epsilon}^{+}(\tau _{\epsilon })$ converges to $0$ in probability as $\eps\to 0$. 
Let us now estimate the integral term. 
Relations~\eqref{eqn: sigma_equal} and~\eqref{eqn: theta+} imply that $(\tau_\eps)_{\eps>0}$ is slowly growing, and we can use 
Lemma~\ref{lemma: y_after_global} to derive
\begin{equation}
\sup_{t\leq \tau_{\epsilon }}|Y _{\epsilon,2 }(t)|=O_{\mathbf{P}}(\epsilon
^{\alpha (1-p)\lambda _{-}/\lambda _{+}}). 
\label{eqn:bound_on_nu_in_probability}
\end{equation}
We can now use Theorem~\ref{Lemma: Def_H1&H2} to conclude that
\begin{equation*}
\epsilon^{-\alpha p} \sup_{t \leq \tau_\epsilon } | H_1 (Y_\epsilon(t),\eps)|=O_{\mathbf{P}} ( \epsilon^{ \alpha(1-p)\lambda _{-}/\lambda _{+} -\alpha p}+\eps^{2-\alpha p}),
\end{equation*}
and (\ref{eqn: p_prop}) implies that the r.h.s.\ converges to $0$. Therefore,

\[
\epsilon^{-\alpha p} \int_{0}^{\tau _{\epsilon }}e^{-\lambda _{+}s} H_{1}(Y_{\epsilon}(s),\eps)ds
\stackrel{\Pp}{\longrightarrow}0.
\]
The above analysis of equation~\eqref{eqn: theta+} implies that if we define $\theta _{0}^{+}=\sgn\eta _{0}^{+}$, then
\begin{align}
 \theta_\eps^+&\stackrel{\mathop{Law}}{\longrightarrow} \theta_0^+,\label{eqn:convergence_of_theta}
\end{align}
which implies~\eqref{eq:tau_eps_convergence} due to~\eqref{eqn: sigma_equal}.
It remains to prove~\eqref{eq:asymptotics_for_Y_2_tau}.

Duhamel's principle along with (\ref{eqn: sigma_equal}) yields
\begin{equation}
Y_{\epsilon,2 }(\tau _{\epsilon })=\left( \frac{|\theta _{\epsilon }^{+}|}{
\delta }\right) ^{\lambda _{-}/\lambda _{+}}\epsilon ^{\alpha \lambda
_{-}/\lambda _{+}}
\eta _{\epsilon }^{-}+\int_{0}^{\tau _{\epsilon
}}e^{-\lambda _{-}(\tau_\eps-s)} H_{2}(Y _{\epsilon}(s),\eps)ds+\epsilon 
N_{\epsilon }^{-}(\tau _{\epsilon }).
\label{eqn:duhamel-at-tau-eps}
\end{equation}
In order to study the convergence of $N_\eps^-(\tau_\eps)$ we first give a preliminary result.
\begin{lemma}
\label{lemma: x_approx}
\begin{equation*}
\sup_{t\leq \tau _{\epsilon }}|Y_{\epsilon,1 }(t)-\epsilon ^{\alpha p}e^{\lambda _{+} t} \sgn\eta _{\epsilon }^{+}| \stackrel{\Pp}{\longrightarrow} 0, \quad \eps \to 0.
\end{equation*}%

\end{lemma}

\begin{proof}
The lemma follows from Duhamel's principle and Lemma~\ref{lemma: x_after_global}.
\end{proof}

The following result is essentially Lemma 8.9 from~\cite{nhn}. It holds true in our
setting since its proof is based only on the conclusion of Lemma~\ref{lemma: x_approx}. 

\begin{lemma}
\label{lemma: stoch_convergence}As $\epsilon \rightarrow 0$, 
\begin{equation*}
N_{\epsilon }^{-}(\tau_{\epsilon })\overset{Law}{%
\longrightarrow }N,
\end{equation*}
where $N$ is the Gaussian random variable in~\eqref{eq:theta}.
\end{lemma}

We finish the proof of Lemma~\ref{Thm:after_tauh}. Recall that the process $M_\eps$ was defined in~\eqref{eqn: M_def} and introduce the stochastic processes
\begin{equation}
R_{\epsilon }(t) = \int_{0}^{t}e^{-\lambda _{-}(t-s)} \hat H_2 ( Y_\eps (s) )ds.  \label{eqn: R_def}
\end{equation}%
Note that~\eqref{eqn:duhamel-at-tau-eps} and \eqref{eqn: sigma_equal}  imply
\begin{align}
 Y_{\eps,2}(\tau_\eps)&=e^{-\lambda_-\tau_\eps}Y_{\eps,2}(0)+\int_0^{\tau_\eps} e^{-\lambda_-(\tau_\eps-s)} H_2(Y_{\epsilon}(s),\eps)ds+\eps N^-_\eps(\tau_\eps)
\notag\\
&=e^{-\lambda _{-}\tau _{\epsilon
}}\epsilon ^{\alpha (1-p)\lambda
_{-}/\lambda _{+}}\eta _{\epsilon }^{-}+\epsilon M_{\epsilon }(\tau_{\epsilon })+R_{\epsilon }(\tau_{\epsilon })  \notag \\
&=\eta _{\epsilon }^{-}\left( \frac{|\theta _{\epsilon }^{+}|}{\delta }%
\right) ^{\lambda _{-}/\lambda _{+}}\epsilon ^{\alpha \lambda _{-}/\lambda
_{+}}+\epsilon M_{\epsilon }(\tau_{\epsilon })+R_{\epsilon }(\tau_{\epsilon }).  \label{eqn: y_order}
\end{align}%
Relations \eqref{eq:condition_for_theorem_along_unstable} and~\eqref{eqn:convergence_of_theta} imply
\begin{equation}
\eta _{\epsilon }^{-}\left( \frac{|\theta _{\epsilon }^{+}|}{\delta }%
\right) ^{\lambda _{-}/\lambda _{+}}\overset{Law}{\longrightarrow }%
\left( \frac{|\eta _{0}^{+}|}{\delta }\right) ^{\lambda _{-}/\lambda
_{+}}y_{2}.
\label{eqn:term_from_initial_cond}
\end{equation}
Lemma~\ref{lemma: stoch_convergence} and estimate~\eqref{eqn:bound_on_nu_in_probability} imply
\begin{equation}
M_{\epsilon }(\tau _{\epsilon })\overset{Law}{\longrightarrow }{N},
\quad\epsilon \rightarrow 0.
\label{eqn:gaussian_term}
\end{equation}
Equations~\eqref{eqn:term_from_initial_cond} and~\eqref{eqn:gaussian_term}
describe the behavior of first two terms in (\ref{eqn: y_order})
and the proof of the lemma will be complete as soon as we show that
\begin{equation}
\epsilon ^{-\beta}R_{\epsilon }(\tau _{\epsilon })\overset{%
\mathbf{P}}{\longrightarrow }0,\quad \epsilon \rightarrow 0.
\label{eqn:R_conv_to_0_in_prob}
\end{equation}%

We can write the following rough estimate based on \eqref{eqn:bound_on_nu_in_probability} and Theorem~\ref{Lemma: Def_H1&H2}:
\begin{equation}
\sup_{t\leq \tau_{\epsilon }}|R_{\epsilon }(t)|=O_{\mathbf{P}%
}(\epsilon ^{2\alpha (1-p)\lambda _{-}/\lambda _{+}}).  \label{eqn: R_est_1}
\end{equation}
This is not sufficient for our purposes. We shall need a more detailed analysis instead.
First, note that 
\begin{equation*}
\sup_{t \leq \tau_\eps }|Y_{\epsilon,2 }(t)-\epsilon M_{\epsilon }(t)-R%
_{\epsilon }(t)|e^{\lambda _{-}t}=\epsilon ^{\alpha (1-p)\lambda
_{-}/\lambda _{+}}|\eta _{\epsilon }^{-}|=O_{\mathbf{P}}(\epsilon ^{\alpha
(1-p)\lambda _{-}/\lambda _{+}}).
\end{equation*}
Hence, for any $\gamma >0$ there is a $K_{\gamma }>0$ such that the event 
\begin{equation*}
D_{\epsilon }=\left\{\sup_{t\leq \tau_{\epsilon }}|Y _{\epsilon,2 }(t)-\epsilon M%
_{\epsilon }(t)-R_{\epsilon }(t)|e^{\lambda _{-}t}<K_{\gamma
}\epsilon ^{\alpha (1-p)\lambda _{-}/\lambda _{+}}\right\}
\end{equation*}%
has probability $\mathbf{P}(D_{\epsilon })> 1-\gamma $ for $%
\epsilon >0$ small enough. Moreover, using Theorem~\ref{Lemma: Def_H1&H2} we see that for some constant $K_\beta>0$,%
\begin{equation*}
|R_{\epsilon }(t)|\leq K_\beta\int_{0}^{t}e^{-\lambda
_{-}(t-s)}Y _{\epsilon,2 }^{2}(s)ds.
\end{equation*}%
Then, using the inequality $(a-b)^{2}\leq 2a^{2}+2b^{2}$ and defining 
$K_{\beta ,\gamma }=K_{\beta }K_{\gamma },$ we see that on  $D_{\epsilon }$ for each $t\leq \tau_{\epsilon }$,  
\begin{align} \notag
|R_{\epsilon }(t)| &
\le K_{\beta }e^{-\lambda_{-} t}\int_{0}^{t}(e^{\lambda
_{-}s}Y _{\epsilon,2 }(s))^2 e^{-\lambda_- s}ds
\\ \notag
& \leq 2K_{\beta ,\gamma }e^{-\lambda
_{-}t}\int_{0}^{t}e^{-\lambda _{-}s}\epsilon ^{2\alpha (1-p)\lambda
_{-}/\lambda _{+}}ds+2K_{\beta }\int_{0}^{t}e^{-\lambda _{-}(t-s)}|\epsilon 
M_{\epsilon }(s)+R_{\epsilon }(s)|^{2}ds \\
&\leq 2\frac{K_{\beta ,\gamma }}{\lambda _{-}}\epsilon ^{2\alpha
(1-p)\lambda _{-}/\lambda _{+}}e^{-\lambda _{-}t}+4\frac{K_{\beta }}{\lambda
_{-}}\epsilon ^{2} M_{\epsilon ,\infty }^{2}+4K_{\beta }e^{-\lambda
_{-}t}\int_{0}^{t}e^{\lambda _{-}s}R_{\epsilon }(s)^{2}ds, \label{eqn:R_eps}
\end{align}%
where  
$M_{\epsilon ,\infty }=\sup_{t\leq \tau _{\epsilon }}|M_{\epsilon
}(t)|,
$
so that (according to Lemma \ref{lemma: stoch_est}) $ M_{\epsilon ,\infty }$ is slowly growing.
Due to \eqref{eqn: R_est_1}  we can find a constant $K_{\gamma }^{\prime }>0$ (independent of $%
\epsilon >0$ and $t>0$) so that the event 
\begin{equation*}
D_{\epsilon }^{\prime }=D_{\epsilon }\cap \left\{\sup_{t\leq
\tau _{\epsilon }}|R_{\epsilon }(t)|\leq K_{\gamma }^{\prime
}\epsilon ^{\alpha (1-p)\lambda _{-}/\lambda _{+}}\right\}
\end{equation*}%
has probability $\mathbf{P}(D_{\epsilon }^{\prime })>1-\gamma $
for all $\epsilon >0$ small enough. Hence, multiplying both sides of~\eqref{eqn:R_eps} by
$e^{\lambda _{-}t}$, we see that for some constant $C_{\gamma }>0$ and all $t\le \tau_\eps$,
\begin{equation*}
e^{\lambda _{-}t}|R%
_{\epsilon }(t)|\mathbf{1}_{\mathcal{D}_{\epsilon }^{\prime }}\leq \alpha (t)+C_{\gamma }\epsilon ^{\alpha (1-p)\lambda
_{-}/\lambda _{+}}\int_{0}^{t}e^{\lambda _{-}s}|R_{\epsilon }(s)|%
\mathbf{1}_{\mathcal{D}_{\epsilon }^{\prime }}ds,
\end{equation*}%
where
\begin{equation}
\alpha (t)=C_{\gamma }\epsilon ^{2\alpha (1-p)\lambda _{-}/\lambda
_{+}}+C_{\gamma }\epsilon ^{2}M_{\epsilon ,\infty }^{2}e^{\lambda _{-}t}. \label{eqn: alpha_def}
\end{equation}
Using Gronwall's lemma and~\eqref{eqn: alpha_def} we get
\begin{eqnarray*}
\mathbf{1}_{\mathcal{D}_{\epsilon }^{\prime }}e^{\lambda _{-}t}|R%
_{\epsilon }(t)| &\leq &\alpha (t)+C_{\gamma }\epsilon ^{\alpha (1-p)\lambda
_{-}/\lambda _{+}}\int_{0}^{t}\alpha (s)e^{C_{\gamma }\epsilon ^{\alpha
(1-p)\lambda _{-}/\lambda _{+}}(t-s)}ds \\
&\leq &\alpha (t)+C_{\gamma }^{2}\epsilon ^{3\alpha (1-p)\lambda
_{-}/\lambda _{+}}t e^{C_{\gamma }\epsilon ^{\alpha (1-p)\lambda _{-}/\lambda
_{+}}t} \\
&&+\frac{C_{\gamma }^{2}}{\lambda _{-}}\epsilon ^{2+\alpha (1-p)\lambda
_{-}/\lambda _{+}}M_{\epsilon ,\infty }^{2} t e^{\lambda _{-}t+C_{\gamma
}\epsilon ^{\alpha (1-p)\lambda _{-}/\lambda _{+}}}.
\end{eqnarray*}%
Hence, 
\begin{eqnarray*}
\mathbf{1}_{\mathcal{D}_{\epsilon }^{\prime }}|R_{\epsilon }(t)|
&\leq &C_{\gamma }\epsilon ^{2\alpha (1-p)\lambda _{-}/\lambda
_{+}}e^{-\lambda _{-}t}(1+C_{\gamma }\epsilon ^{\alpha (1-p)\lambda
_{-}/\lambda _{+}}t e^{C_{\gamma }\epsilon ^{2\alpha (1-p)\lambda _{-}/\lambda
_{+}}t}) \\
&&+C_{\gamma }\epsilon ^{2}M_{\epsilon ,\infty }^{2}(1+\frac{C_{\gamma }}{%
\lambda _{-}}\epsilon ^{\alpha (1-p)\lambda _{-}/\lambda _{+}}t e^{C_{\gamma
}\epsilon ^{\alpha (1-p)\lambda _{-}/\lambda _{+}}}).
\end{eqnarray*}%
Using \eqref{eqn: sigma_equal}, we get that for any $q>0$,
\begin{align*}
\mathbf{1}_{\mathcal{D}_{\epsilon }^{\prime }}|R_{\epsilon
}(\tau _{\epsilon })|
&= O_\Pp \left(\epsilon ^{2\alpha (1-p)\lambda _{-}/\lambda
_{+}}e^{-\lambda _{-}\tau_\eps}+\epsilon ^{2}M_{\epsilon
,\infty }^{2}\right)
\\&
=O_\Pp\left(\epsilon ^{\alpha \lambda _{-}/\lambda _{+}+\alpha (1-p)\lambda
_{-}/\lambda _{+}}+\epsilon ^{2-q}\right),
\end{align*}%
so that~\eqref{eqn:R_conv_to_0_in_prob} follows, and the proof is complete by choosing $q$ small enough.

\bibliography{happydle}{}
\bibliographystyle{plain}
\end{document}